\documentclass[11pt]{amsart}
\usepackage{tabularx,booktabs,tikz}
\usepackage{caption}
\usepackage{amsmath}
\usepackage{amsfonts}

\usepackage{amscd}
\usepackage{amsthm}
\usepackage{amssymb} 
\usepackage{latexsym}
\usepackage{eufrak}
\usepackage{euscript}
\usepackage{epsfig}
\usepackage{graphics}
\usepackage{array}
\usepackage{enumerate}
\usepackage{dsfont}
\usepackage{color}
\usepackage{wasysym}
\usepackage{hyperref}
\usepackage{pdfsync}

\newcommand{\Hmm}[1]{\leavevmode{\marginpar{\tiny%
$\hbox to 0mm{\hspace*{-0.5mm}$\leftarrow$\hss}%
\vcenter{\vrule depth 0.1mm height 0.1mm width \the\marginparwidth}%
\hbox to
0mm{\hss$\rightarrow$\hspace*{-0.5mm}}$\\\relax\raggedright #1}}}

\newtheorem{theorem}{Theorem}[section]
\newtheorem{lemma}[theorem]{Lemma}
\newtheorem{corollary}[theorem]{Corollary}
\newtheorem{definition}[theorem]{Definition}

\newtheorem{remark}[theorem]{Remark}

\newtheorem{proposition}[theorem]{Proposition}

\begin{document}

\title[Ground state solutions to some Indefinite Nonlinear Schr\"{o}dinger equations on lattice graphs]{Ground state solutions to some indefinite variational problems on lattice graphs}


\author{Wendi Xu}
\address{Wendi Xu: School of Mathematical Sciences, Fudan University, Shanghai 200433, China}
\email{wdxu19@fudan.edu.cn}

\begin{abstract}
In this paper, we consider the Schrödinger type equation 
$-\Delta u+V(x)u=f(x,u)$ on the lattice graph $\mathbb{Z}^{N}$ with indefinite variational functional, where $-\Delta$ is the discrete Laplacian. Specifically, we assume that $V(x)$ and $f(x,u)$ are periodic in $x$, $f$ satisfies some growth condition and 0 lies in a spectral gap of $(-\Delta + V)$. We obtain ground state solutions by using the method of generalized Nehari manifold which has been introduced in \cite{MR2768820}. 

\end{abstract}
\par
\maketitle

\bigskip

\section{introduction}
Nowadays, mathematicians pay attention to the analysis on graphs which is important for practical applications. They have obtained various results analogous to the classical theory in Euclidean spaces and manifolds, see for examples \cite{MR1743100,MR3616731,MR3822363,MR4383783,MR3265330,MR3316971,MR3581303, MR3776357,MR4036571} and the references therein. Difference equations on graphs are discrete counterparts of partial differential equations. 
In \cite{MR3523107,MR3542963,MR3665801}, the authors use the variational methods to study Kazdan-Warner equations, Yamabe equations and Schr\"{o}dinger equations on locally finite graphs respectively. 
The Schr\"{o}dinger type equation 
\begin{equation}\label{eq}
	-\Delta u +V(x)u = f(x,u),\quad x\in \Omega,
\end{equation}
where $\Omega \subset \mathbb{R}^{n},\,n\geq 2$, has been extensively studied in the Euclidean spaces. See for examples \cite{MR709644, MR1163431, MR1162728, MR2271695, MR2557725, MR2873855} and the references therein. 
We are interested in the equation (\ref{eq}) with periodic potential, which has been studied in many papers, such as \cite{MR1181725, MR850686, MR1026427, MR1205391, MR2175045}.
 In \cite{MR2271695}, the authors consider two cases of potentials $V$; one is positive periodic and the other is positive and bounded. Using the Nehari method, they find ground state solutions for (\ref{eq}) without compact embeddings. The case when the operator $-\Delta+V$ is indefinite is considered in \cite{MR2557725}. The authors, via the method of generalized Nehari manifold, reduce the indefinite variational problem to a definite one and give a new characterization of the corresponding critical value. 
There have been many works studying Schr\"{o}dinger type equations on graphs as well. In \cite{MR3665801}, the authors obtain a strictly positive solution for the equation (\ref{eq}) on locally finite graphs with $V(x)\geq V_{0}>0$, $1/V \in L^{1}(\mathbb{V})$ and some additional assumptions. In \cite{MR3833747}, N. Zhang and L. Zhao  prove via the Nehari method that, provided $a^{-1}(0) \neq \emptyset$ and $a(x)\to \infty$ as $d(x,x_{0})\to\infty$ for fixed $x_{0}$, the equation (\ref{eq}) with $V(x)=\lambda a(x) +1$ admits a ground state solution $u_{\lambda}$ for every $\lambda >1$. Besides, $u_{\lambda}$ converges to a solution for the Dirichlet problem $-\Delta u + u = |u|^{p-1}u$ as $\lambda \to \infty$. For some other works related to  Schr\"{o}dinger equations on graphs we refer to \cite{MR4092834, MR4053610, MR3921107} for examples.
Note that both \cite{MR3665801} and \cite{MR3833747} require the potential $V(x)$ tending to infinity as $d(x,x_{0})\to \infty$ in order to
obtain compact Sobolev embeddings while we are concerned with periodic potentials in this paper.
\\ \hspace*{\fill}\\
We first introduce the basic setting of lattice graphs. Let $\mathbb{Z}^{N}=(\mathbb{V},\mathbb{E})$ be a lattice graph with the set of vertices
$$ \mathbb{V}:=\{x=(x_{1},\cdots,x_{N}):x_{i}\in \mathbb{Z}, 1 \leq i \leq N\}$$ and the set of edges
$$
\mathbb{E}:=
\Big\{
 \{ x,y\}: x,y\in\mathbb{V},\sum_{i=1}^{N}|x_{i}-y_{i}|=1
 \Big\}.
 $$ 
Two vertices $x$ and $y$ are called neighbours, denoted by $x\sim y$, if there is an edge $ \{ x,y\}\in \mathbb{E}$.
 Let $w:\mathbb{E} \to \mathbb{R}^{+}$ and $\mu: \mathbb{V} \to \mathbb{R}^{+}$ be the measures on the edge set and the vertex set respectively. In this paper, we consider the lattice graph with unit weight, i.e.,
\begin{displaymath}
	\begin{aligned}
		&w_{xy} = 1, \quad \forall  \, \{x,y\} \in \mathbb{E} , \\
		&\mu(x)=1,\quad \forall  \, x\in \mathbb{V}.
	\end{aligned}
\end{displaymath}
  For each $x\in \mathbb{V}$, the degree of $x$ is the number of its neighbours
 $	d_{x}:=\sum_{
	\substack{y\in\mathbb{V}\\
		y\sim x}
}  w_{xy}=2N.$
Denote the set of functions on $\mathbb{Z}^{N}$ by $C(\mathbb{V})$. Define the combinatorial Laplacian as, for any function $u\in C(\mathbb{V})$ and $x\in \mathbb{V}$,
\begin{displaymath}
		\Delta u(x) 
		:=\sum_{y \sim x} \big(u(y)-u(x)\big).
\end{displaymath}
For any function $u,v\in C(\mathbb{V})$, define the associated gradient form as
\begin{displaymath}	
			\Gamma (u,v)(x) 
			:= \frac{1}{2} \sum_{y \sim x} \big(u(y)-u(x)\big)\big(v(y)-v(x)\big).
\end{displaymath}
Set $\Gamma(u) = \Gamma(u,u)$ and the length of its gradient reads
\begin{displaymath}
	| \nabla u |(x) := \sqrt{\Gamma (u)(x)} =
    \Big(
	\frac{1}{2} \sum_{y \sim x} \big(u(y)-u(x)\big)^{2} \Big)^{1/2}.
\end{displaymath}
For any function $f\in C(\mathbb{V})$, we write
\begin{displaymath}
	\int_{\mathbb{V}} f d \mu := \sum_{x\in \mathbb{V}} \mu(x)f(x) = \sum_{x\in \mathbb{V}} f(x) 
\end{displaymath}
whenever it makes sense.
Let $C_{c}(\mathbb{V})$ be the set of all functions with finite support, and $W^{1,2}(\mathbb{V})$ be the completion of $C_{c}(\mathbb{V})$ under the norm
\begin{displaymath}
	\|u\|_{W^{1,2}}:=
	 \Big( 
	\int_{\mathbb{V}} (| \nabla u|^{2} + u^{2} ) d\mu
	 \Big)^{1/2}.
\end{displaymath}
One can check that $W^{1,2}(\mathbb{V})$ is a Hilbert space with the inner product
\begin{equation}\label{inner p.}
 \langle u,v \rangle _{W^{1,2}} = 
 \int_{\mathbb{V}} (\Gamma(u,v)+uv) d\mu, 
 \quad \forall u,v \in W^{1,2}.
\end{equation}
Let $\ell^{p}(\mathbb{V},\mu), p\in[1,\infty],$ be the space of $\ell^{p}$ summable functions on $\mathbb{V}$ w.r.t. the measure $\mu$. We write $\|\cdot\|_{p}$ as the $\ell^{p}(\mathbb{V},\mu)$ norm, i.e., 
\begin{equation}\notag
	\|u\|_{p}:=
	\left\{ 
	\begin{aligned}
		& \big(  \int_{\mathbb{V}} |u(x)|^{p}  d\mu \big)^{\frac{1}{p}},\quad 1\leq p < \infty,\\
		&\sup_{x\in\mathbb{V}} |u(x)|,\quad \quad \quad\quad p=\infty.
	\end{aligned}
	\right.
\end{equation}
For $T\in \mathbb{N}$, $g\in C(\mathbb{V})$ is called T-periodic if $$g(x+Te_{i})=g(x),\quad \forall x\in\mathbb{V}, 1\leq i \leq N,$$ where $e_{i}$ is the unit vector in the i-th coordinate.
\\ \hspace*{\fill} \\
We are concerned with the existence of ground state solutions of the following Schr\"{o}dinger equation:
\begin{equation}\label{maineq}
	\left \{\begin{array}{ll}
		-\Delta u+V(x)u=f(x,u), & \\	
		u\in W^{1,2}(\mathbb{V}).&\,
	\end{array}\right. 
\end{equation}
For the Schr\"{o}dinger operator $-\Delta + V$, suppose that the following hold:\\
$(S_{1})$  $V(x)$ is a bounded T-periodic function and 0 lies in a gap of the spectrum of $-\Delta + V$ .\\
$(S_{2})$  $f:\mathbb{V}\times\mathbb{R}\to \mathbb{R}$  is continuous, T-periodic in $x$ and satisfies the growth condition 
\begin{equation}\label{g.c.}
	|f(x,u)| \leq a(|u|+|u|^{p-1}) \quad \textrm{ for some }  a>0 \textrm{ and } p>2.
\end{equation} 
$(S_{3})$ \,$f(x,u) = o(u)$ uniformly in $x$ as $u \to 0$.\\ 
$(S_{4})$  \,$ \frac{F(x,u)}{u^2} \to \infty $ uniformly in $x$ as $|u| \to \infty $ with $F(x,u):= \int_{0}^{u} f(x,s) \,d s $.\\
$(S_{5})$  \,$u\mapsto \frac{f(x,u)}{|u|}$ is strictly increasing on $(-\infty,0)$ and $(0,+\infty)$.\\ 

\begin{theorem} \label{thm1}
Let $\mathbb{Z}^{N}=(\mathbb{V},\mathbb{E})$ be a lattice graph with unit weight. Suppose that $(S_{1}-S_{5})$ hold. Then (\ref{maineq}) admits a ground state solution.
\end{theorem} 	
This theorem is a discrete analog of the result in \cite{MR2557725}. However, the embeddings in our setting are different from that, which allows us to remove the constraint $p<\frac{2N}{N-2}$(subcritical). We only need to assume $p>2$.

The multiplicity of solutions to (\ref{maineq}) is very important, see \cite{MR2557725} for details. 

Denote the orbit of $\,u_{0}$ under the action of $\, \mathbb{Z}^{N}$ by $\,\mathcal{O}(u_{0})$, i.e., $\mathcal{O}(u_{0}) := \{u_{0}(\cdot-kT):k\in\mathbb{Z}^{N}\}$. Note that if $u_{0}$ is a solution of ($\ref{maineq}$), then so  are all elements in $\mathcal{O}(u_{0})$. We say two solutions $u_{1}$ and $u_{2}$ are geometrically distinct if $\mathcal{O}(u_{1})$ and $\,\mathcal{O}(u_{2})$ are disjoint.
\begin{theorem} \label{thm2}
	Suppose that $f$ is odd in u and the assumptions in Theorem \ref{thm1} are satisfied. Then (\ref{maineq})  admits infinitely many geometrically distinct solutions.	
\end{theorem} 
This paper is organized as follows. In Section 2, we prove the equivalence and embeddings between some spaces. Moreover, we decompose $W^{1,2}$ corresponding to the spectral decomposition of $-\Delta + V$ and introduce an equivalent inner product in $W^{1,2}$. The generalized Nehari manifold is introduced in Section 3 and we also establish the variational framework in this part. The proof of Theorem \ref{thm1} is given in Section 4. In Section 5, we talk about the idea of the proof of Theorem \ref{thm2}. The Appendix is a supplement to the proof of Proposition \ref{prop1}.

\section{Preliminaries}
Let $\mathbb{Z^{N}}=(\mathbb{V},\mathbb{E})$ be a lattice graph.
 We will show that 
$\| \cdot \|_{W^{1,2}} $ and $ \| \cdot \|_{2}$ are equivalent norms on $C(\mathbb{V})$, denoted by 
$\| \cdot  \|_{W^{1,2}}   \sim  \| \cdot  \|_{2}   $, i.e., there exist constants $C_{1}, C_{2} >0$, such that 
 $$C_{1}\| u \|_{2}  \leq  \| u \|_{W^{1,2}}  \leq  C_{2} \| u \|_{2} $$ 
 for all $u\in C(\mathbb{V})$.
Note that, $\forall u\in W^{1,2}(\mathbb{V})$,
\begin{displaymath}
	\begin{aligned}
		\mathcal{E} (u) := &\int_{\mathbb{V}} 	| \nabla u |^{2} d \mu
		=\sum_{x\in \mathbb{V}}\sum_{y \sim x} \frac{1}{2}\big(u(y)-u(x)\big)^{2}\\
		\leq & \sum_{x\in \mathbb{V}}\sum_{y \sim x} \big(u^{2}(y)+u^{2}(x)\big)
		=  2 \sum_{x\in \mathbb{V}}\sum_{y \sim x}  u^{2}(x)\\
			=&4N\|u\|_{2}^{2}.
		\end{aligned}
	\end{displaymath}
We have $\|u\|^{2}_{2} \leq \|u\|^{2}_{W^{1,2}} \leq (4N+1)\|u\|^{2}_{2}$ and then we can regard the two spaces $W^{1,2}(\mathbb{V})$ and $\ell^{2}(\mathbb{V})$ as the same.
Clearly, 
$\ell^{p}(\mathbb{V}) \hookrightarrow \ell^{q}(\mathbb{V})$ continuously for any $2 \leq p<q \leq \infty$. Hence,  we have 
\begin{displaymath}
	\begin{aligned}
		&W^{1,2}(\mathbb{V}) \sim \ell^{2}(\mathbb{V}),\\
		&W^{1,2}(\mathbb{V}) \hookrightarrow \ell^{p}(\mathbb{V}), \quad \forall p > 2.
	\end{aligned}
\end{displaymath}
Let $E:=W^{1,2}(\mathbb{V})$ with the inner product $ \langle \cdot,\cdot \rangle_{W^{1,2}} $ defined as ($\ref{inner p.}$). 
Set $L:= -\Delta +V$. This is a self-adjoint operator on $E$ under the assumption $(S_{1})$. By the spectral decomposition theorems of self-adjoint operators in \cite{MR840880}, there exists a pedigree $\{ F_{\lambda}| \lambda \in (-\infty,+\infty)\}$ in $E$ such that
$$ L = \int_{-\infty}^{+\infty} \lambda dF_{\lambda}
 \equiv \int_{\sigma(L)} \lambda dF_{\lambda}.$$
Set 
$$T^{+}:=\int_{\sigma(L)} \mathcal{X}_{\sigma(L)\cap(0,+\infty)} (\lambda) d F_{\lambda},$$
$$T^{-}:=\int_{\sigma(L)} \mathcal{X}_{\sigma(L)\cap(-\infty,0)} (\lambda) dF_{\lambda}.$$
They are both projection operators by the definition of pedigree. Moreover, 
\begin{equation}\label{identity}
	T^{+}+T^{-} = I
\end{equation}
since $\int_{\sigma(L)} 1 dF_{\lambda} = I$. Define two subspaces of $E$ as the following:
$$E^{+} := \{ u\in E | T^{+}u = u \},$$
$$E^{-} := \{ u\in E | T^{-}u = u \}.$$
According to (\ref*{identity}) and the properties of projection operators, it's easy to check $E=E^{-} \oplus E^{+}$. For  each $u \in E^{+}$,
\begin{displaymath}
	\begin{aligned}
		(Lu,u)&=(LT^{+}u,u)\\
		& =\int_{\sigma(L)} \lambda\mathcal{X}_{\sigma(L)\cap(0,+\infty)} (\lambda) d (F_{\lambda}u,u)\\
		&=\int_{\sigma(L)\cap(0,+\infty)} \lambda d(F_{\lambda}u,u) \geq 0,
	\end{aligned}
\end{displaymath}
where the last inequality is based on the fact that $F_{\lambda}$ is a projection operation and $(\cdot,\cdot)$ is the standard inner product in $l^{2}(\mathbb{V})$. 
Similarly, we have $(Lu,u)\leq 0$ for each $u \in E^{-}$. That's to say, 	$L$ is positive definite on $E^{+}$ and negative definite on $E^{-}$. Then we can define new inner products $\langle\cdot,\cdot \rangle$ on $E^{+}$ and $E^{-}$ respectively: 
\begin{displaymath}
	\begin{aligned}
		&\langle u,v \rangle := (Lu,v),\quad \forall u,v \in E^{+} ,\\
		&\langle u,v \rangle := -(Lu,v),\quad \forall u,v \in E^{-} .\\
	\end{aligned}
\end{displaymath}
Correspondly, we have new norms $\|\cdot\|$ induced by the inner products. It is well-known that the spectrum of combinatorial Laplacian on $\mathbb{Z}^{N}$ is $[0,4N]$, see \cite{MR1782338} for example. Note that $L$ is a bounded self-adjoint operator, being negative and positive definite  on $E^{-}$ and $E^{+}$ respectively. Taking $E^{-}$ as an example, there exist constants $\alpha>\beta>0$ such that
\begin{displaymath}
	\begin{aligned}
		-\alpha 
		&= \inf_{\lambda \in \sigma(L)\cap (-\infty,0)} \lambda
		=\inf_{
			          \substack{\|u\|_{W^{1,2}}=1 \\
			                               u\in E^{-} }
			         } (Lu,u),	\\
		-\beta 
		&= \sup_{\lambda \in \sigma(L)\cap (-\infty,0)} \lambda
		=\sup_{
			            \substack{\|u\|_{W^{1,2}}=1 \\
				                           u\in E^{-} }
	              	} (Lu,u).	\\
	\end{aligned}
\end{displaymath}
For each $ u \in E^{-}$, 
\begin{displaymath}
	\begin{aligned}
		-\alpha &\leq \frac{(Lu,u)}{\|u\|^2_{W^{1,2}}} \leq -\beta,\\
		\beta \|u\|^2_{W^{1,2}} &\leq -(Lu,u) \leq \alpha \|u\|^2_{W^{1,2}},
	\end{aligned}
\end{displaymath}
i.e., 
$$	\beta \|u\|^2_{W^{1,2}} \leq \|u\|^{2} \leq \alpha\|u\|^2_{W^{1,2}}.$$
So, $\| \cdot \| \sim \| \cdot \|_{W^{1,2}}$ on $E^{-}$. Similarly, one can check $\| \cdot \| \sim \| \cdot \|_{W^{1,2}}$ on $E^{+}$.
Moreover, $\forall u=u^{+}+u^{-} \in E^{+}  \oplus E^{-}$,
\begin{displaymath}
	\begin{aligned}
		\|u\|^{2} &= \|u^{+}\|^{2} + \|u^{-}\|^{2} ,\\
		\|u \|^{2}_{W^{1,2}} &= \|u^{+}\|_{W^{1,2}}^{2} + \|u^{-}\|_{W^{1,2}}^{2} .
	\end{aligned}
\end{displaymath}
 It follows from above that $\| \cdot \| \simeq \| \cdot \|_{W^{1,2}}$  on $E$.

Each weak solution of (\ref{maineq}) corresponds to a critical point of the following variational functional:
\begin{displaymath}
	\begin{aligned}
			\Phi(u) &= \frac{1}{2} \int_{\mathbb{V}} 
			|\nabla u|^{2}+V(x)u^{2} d\mu 
			-\int_{\mathbb{V}} F(x,u) d\mu \\ 
	&	\equiv \frac{1}{2}(\|u^{+}\|^{2}-\|u^{-}\|^{2})
		-\int_{\mathbb{V}} F(x,u) d\mu.
	\end{aligned}
\end{displaymath}
 
  \begin{remark}\label{rm1}
 	By choosing appropriate test functions, one can see that each weak solution actually satisfies the equation $(\ref{maineq})$ pointwisely. See Proposition3.1 in \cite{MR3665801}.
 \end{remark}

\section{Generalized Nehari Manifolds}
	Let $u$ be a nontrivial critical point of $\Phi$, i.e., $\Phi'(u)=0$ and $u\neq 0$, then $$\Phi(u)=\Phi(u)-\frac{1}{2}\Phi'(u)u 
	= \int_{\mathbb{V}}  \frac{1}{2}f(x,u)u - F(x,u) d\mu.$$
Note that, according to $(S_{2})$ and $(S_{3})$,
\begin{equation}\label{uf(x,u)}
0<F(x,u)<\frac{1}{2}f(x,u)u, \quad\forall u\neq0.
\end{equation} 
Every nontrivial critical point $u_{0}$ corresponds to a positive value $\Phi(u_{0})>0$. We call it a ground state solution if it corresponds to the minimal critical value.\\
Define
\begin{displaymath}
		\mathcal{M}:= \{ u\in E\backslash E^{-} : \Phi'(u)u=0 \text{ and } \Phi'(u)v=0 ,\forall v\in E^{-}\}.
\end{displaymath}
$\mathcal{M}$ is called the generalized Nehari manifold which has been introduced by \cite{MR2175045}. We show that $\mathcal{M}$ is homeomorphic to the unit sphere $S^{+}$ in $E^{+}$. In fact, $\mathcal{M}$ contains all nontrivial critical points of $\Phi$ because
 $\Phi \leq 0$ on $ E^{-}$.
 Set $c:= \inf_{\mathcal{M}}\Phi.$
 We hope $c$ is attained at some $u_{0}\in \mathcal{M}$ and $u_{0}$ is a critical point of $\Phi$.
For each $u \in E\backslash E^{-}$, set
\begin{displaymath}
	\begin{aligned}
		E(u) &:= E^{-} \oplus \mathbb{R} u = E^{-} \oplus \mathbb{R} u^{+}, \\
		\hat{E}(u) &:= E^{-} \oplus \mathbb{R}^{+} u = E^{-} \oplus \mathbb{R}^{+} u^{+},
	\end{aligned}
\end{displaymath}
where $\mathbb{R}^{+} :=[0,+\infty)$.
In the following, we use the method of generalized Nehari manifold to find a critical point of $\Phi$. Specifically, we first establish a homeomorphism between $\mathcal{M}$ and $S^{+}$. Then, the original problem is transformed into finding a critical point of a $C^{1}$ functional on $S^{+}$. It's worth pointing out that $S^{+}$ is a $C^{1}$ submanifold.

\begin{lemma}\label{lem2}
	For each $\epsilon > 0$, there exists $C_{\epsilon} >0$ such that 
	$$
	|f(x,u)| \leq \epsilon |u| + C_{\epsilon}|u|^{p-1},\quad \forall u\in \mathbb{R}. 
	$$
\end{lemma}
This follows from $(S_{2})$ and $(S_{3})$.
\begin{proposition}\label{prop1}
	If  $\,u \in \mathcal{M}$, then $u$ is the unique global maximum of $\Phi |_{\hat{E}(u)}$.
\end{proposition}
The proof of Proposition \ref{prop1} is similar to the continuous case, see \cite{MR2557725}, and we put it in the Appendix.
\begin{lemma}\label{lem4}
	(i) There exists $\alpha > 0$ such that $c=\inf_{\mathcal{M}} \Phi \geq \inf_{S_{\alpha}} \Phi > 0 $, where $S_{\alpha}:=\{u\in E^{+}:\|u\|=\alpha\}$.\\
	(ii) For each $u\in \mathcal{M}$, $\|u^{+}\|\geq $ max$\{\|u^{-}\|,\sqrt{2c} \}$. 
\end{lemma}

\begin{proof}
(i)For every $u\in \mathcal{M}$, there exists $s>0$ such that $su^{+} \in \hat{E}(u) \cap S_{\alpha}$. By Proposition \ref{prop1}, $\Phi(u) \geq \Phi(su^{+} ) \geq \inf_{S_{\alpha}}\Phi$. Take the infimum among all $u$ in $\mathcal{M}$ and we have the first inequality. For $u\in E^{+}$,  $\Phi(u)=\frac{1}{2}\|u\|^{2}-\int_{\mathbb{V}} F(x,u) d\mu$. 
By Lemma \ref{lem2}, 
$\int_{\mathbb{V}} F(x,u) d\mu = o(\|u\|^{2})$ as $u\rightarrow 0$. So, there exists $\alpha >0$ small enough such that the second inequality holds.

(ii)For each $u\in \mathcal{M}$, 
$$
0<c \leq \tfrac{1}{2}(\|u^{+}\|^{2}-\|u^{-}\|^{2})-\int_{\mathbb{V}}F(x,u) d\mu \leq \tfrac{1}{2}(\|u^{+}\|^{2}-\|u^{-}\|^{2}).
$$ 
\end{proof}

\begin{lemma}\label{lem5}
Let $\mathcal{V} \subset E^{+}\backslash \{0\}$ be a compact subset, then there exists $R>0$ such that $\Phi \leq 0$ on $E(u)\backslash B_{R}(0)$ for every $u \in \mathcal{V}$.
\end{lemma}

\begin{proof}
	Since $E(u)=E(\frac{u}{\|u\|})$, we may assume that $\|u\| = 1$ for every $u \in \mathcal{V}$. Suppose by contradiction that there exist $u_{n} \in \mathcal{V}, w_{n} \in E(u_{n})$ such that $\Phi(w_{n}) >0$ for all $n\in \mathbb{N}$ and $\|w_{n}\| \rightarrow \infty \text{ as } n \rightarrow \infty$. Set $v_{n}:=\tfrac{w_{n}}{\|w_{n}\|}= s_{n}u_{n}+v_{n}^{-}$ with $s_{n}^{2} + \|v_{n}^{-}\|^{2} =1$.
Then
\begin{equation}\label{*}
	0<\dfrac{\Phi(w_{n})}{\|w_{n}\|^{2}}
	=\dfrac{1}{2}(s_{n}^{2}-\|v_{n}^{-}\|^{2}) 
	-\int_{\mathbb{V}} \dfrac{F(x,w_{n})}{w_{n}^{2}(x)}v_{n}^{2}(x) d\mu.  
\end{equation}
$\|v_{n}^{-}\|^{2} < s_{n}^{2} = 1-\|v_{n}^{-}\|^{2} $ implies $\dfrac{1}{\sqrt{2}} < s_{n} < 1$ and we have, up to a subsequence, $s_{n} \rightarrow s \in [\dfrac{1}{\sqrt{2}} ,1]$. After passing to a subsequence, let $v_{n} \rightharpoonup v$ in $(E,\| \cdot \| )$. For $x \in \mathbb{V}$, choose a test function $\phi = \delta_{x} \in C_{c}(\mathbb{V})$ and we have $v_{n}(x) \rightarrow v(x)$ pointwisely. It follows that $v=su+v^{-} \neq 0$, i.e., $\exists x_{0} \in \mathbb{V}$, s.t. $v(x_{0}) \neq 0$ and $w_{n}(x_{0})=\|w_{n}\|v_{n}(x_{0})\rightarrow +\infty$ as $n \rightarrow \infty$. By Fatou's lemma,
$$
\int_{\mathbb{V}} \dfrac{F(x,w_{n})}{w_{n}^{2}}v_{n}^{2} d\mu \rightarrow +\infty ,
$$
which contradicts (\ref{*}).
\end{proof} 

\begin{proposition}\label{prop2}
	For each $u\notin E^{-}$,  $\hat{E}(u)$ and $\mathcal{M}$ have a unique intersection, denoted by $\hat{m}(u)$ and this $\hat{m}(u)$ is the unique global maximum of $\Phi|_{\hat{E}(u)}$.
\end{proposition}

\begin{proof}
By Proposition \ref{prop1}, $\hat{m}(u) \in \mathcal{M}$ is the unique global maximum of $\Phi|_{\hat{E}(\hat{m}(u))}$. Moreover,  $\hat{E}(u)=\hat{E}(\hat{m}(u))$ by definition. So, it suffices to show that $\mathcal{M} \cap \hat{E}(u) \neq \emptyset$. \\
	For each $u\in E\backslash E^{-}$, since $\hat{E}(u) = \hat{E}(u^{+})=\hat{E}\big(\tfrac{u^{+}}{\|u^{+}\|}\big),$
	we may assume that $u\in E^{+}$ and $\|u\|=1$.
By Lemma \ref{lem4}(i), there exists $t>0$ small enough such that $\Phi(tu)>0$ . According to Lemma \ref{lem5}, $\Phi \leq 0$ on $\hat{E}(u)\backslash B_{R}(0)$ for some $R>0$. It follows that $0<\sup_{\hat{E}(u)} \Phi <\infty.$ \\
Set $\{u_{n}\} \subset \hat{E}(u) \cap B_{R}(0) \text{ with }  \Phi(u_{n}) \rightarrow \sup_{\hat{E}(u)} \Phi .$
 Then, up to a subsequence, we have $u_{n} \rightharpoonup u_{0}$.  Since $\hat{E}(u)$ is weakly closed as a closed convex subset of Banach space $E$, $u_{0} \in \hat{E}(u)\backslash \{0\}$. By Fatou's lemma and the dominated convergence theorem, $\Phi$ is weakly upper semicontinuous on $\hat{E}(u)$, i.e., 
$\lim_{n \rightarrow \infty} \Phi(u_{n}) \leq \Phi(u_{0})$ for $u_{n}\rightharpoonup u_{0}$. So, $\Phi(u_{0})=\sup_{\hat{E}(u)}\Phi$.  $u_{0}$ is a critical point of $\Phi|_{E(u)}$ implies $\Phi'(u_{0})u_{0}=\Phi'(u_{0})v=0$ for any $v \in E(u)$.  It follows that $u_{0}\in \mathcal{M}\cap \hat{E}(u)$.
\end{proof}

\begin{remark}
(i)	We have proved a new characterization for the least positive energy of $\Phi$ on $\mathcal{M}$ :
	$$
	0 < c = \inf_{\mathcal{M}} \Phi = \inf_{w\in E^{+} \backslash \{0\}} \sup_{u\in\hat{E}(w)} \Phi(u).
	$$
(ii)Define 
\begin{displaymath}
	\begin{aligned}
		\widehat{m} : E^{+}\backslash\{0\} &\to \mathcal{M} \\
		u &\mapsto 	\widehat{m}(u).
	\end{aligned}
\end{displaymath}
We'll prove in the following that $\hat{m}$ is a continuous map.
\end{remark}

\begin{proposition}\label{prop3}
	$\Phi$ is coercive on $\mathcal{M}$, i.e., $\Phi(u) \rightarrow \infty \text{ as } \|u\| \rightarrow \infty, u\in \mathcal{M}$.
\end{proposition}

\begin{proof}
	Suppose the proposition does not hold and let $\{u_{n}\} \subset \mathcal{M}$ with $$\|u_{n}\| \rightarrow \infty \text{ and } \Phi(u_{n}) \leq d $$
 for some $d\geq c$. Set $v_{n} :=\frac{u_{n}}{\|u_{n}\|}$, 
we may assume that, after passing to a subsequence, $v_{n}\rightharpoonup v$ and $v_{n}(x) \rightarrow v(x), \forall x\in \mathbb{V}$. Suppose $v_{n}^{+} \rightarrow 0$ in $l^{p}(\mathbb{V})$, $p>2$. By Lemma \ref{lem2}, $\int_{\mathbb{V}} F(x,sv_{n}^{+}) d\mu \rightarrow 0$ as $n\rightarrow \infty$ for every $s>0$. It follows that
\begin{equation}\label{^}
	\Phi(sv_{n}^{+}) = \dfrac{1}{2} s^{2}\|v_{n}^{+}\|^{2} - \int_{\mathbb{V}} F(x,sv_{n}^{+}) d\mu \geq \dfrac{1}{4} s^{2} \quad \text{ as } n \rightarrow \infty.
\end{equation}
Moreover, by Proposition \ref{prop1},
\begin{equation}\label{^*}
	\Phi(sv_{n}^{+}) \leq \Phi(u_{n}) \leq d.
\end{equation}
 (\ref{^}) contradicts (\ref{^*}) whenever $s>2\sqrt{d}$. It follows that $v_{n}^{+} \nrightarrow 0$ in $l^{p}(\mathbb{V})$, i.e., $$\varlimsup_{n} \|v_{n}^{+}\|_{p} = \beta$$ for some $\beta >0$. By interpolation inequality, 
 $$
 \beta = \limsup_{n}\|v_{n}^{+}\|_{p} \leq \limsup_{n}\|v_{n}^{+}\|_{2}^{\frac{2}{p}} \cdot \|v_{n}^{+}\|_{\infty}^{\frac{p-2}{p}}.
 $$
 Since $\|v_{n}^{+}\|_{2} \leq C\|v_{n}^{+}\| \leq A$ for some constant $A>0$, 
 $$
 \limsup_{n} \|v_{n}^{+}\|_{\infty} \geq \big(\frac{\beta^{p}}{A^{2}}\big)^{\frac{1}{p-2}} >0.
 $$
 Therefore, there exists a subsequence $\{v_{n}\}$ and a sequence $\{y_{n}\} \subset \mathbb{V}$ such that $|v_{n}^{+}(y_{n})| \geq \big(\frac{\beta^{p}}{A^{2}}\big)^{\frac{1}{p-2}} $ for each $n$. 
  By translations, set $\tilde{v}_{n}(y):=v_{n}(y+k_{n}T)$ with $k_{n}=(k^{1}_{n},\dots k^{N}_{n})$ to ensure that $\{y_{n}-k_{n}T\} \subset \Omega$ where $\Omega=[0,T)^{N}\bigcap \mathbb{V}$ is a bounded domain in $\mathbb{V}$. For each $n$,
 $$\|\tilde{v}_{n}^{+}\|_{\ell^{\infty}(\Omega)} \geq |v_{n}^{+}(y_{n})| \geq \big(\frac{\beta^{p}}{A^{2}}\big)^{\frac{1}{p-2}}  >0.$$
  Correspondingly, translate $u_{n}$ to  $\tilde{u}_{n}(y):=\|u_n\|\tilde{v}_{n}(y) \equiv u_{n}(y+k_{n}T)$. Since $\Omega$ is bounded, there exists at least one point, say $x_{0}\in \Omega $, such that $\tilde{v}_{n}(x_{0}) \to \tilde{v}(x_{0}) >0$ and $|\tilde{u}_{n}(x_{0})|=\|\tilde{u}_{n}\||\tilde{v}_{n}(x_{0})| \to \infty$ as $n \to \infty$. By $(S_{4})$, 
  $$\int_{\mathbb{V}} \frac{F(x,\tilde{u}
  	_{n})}{\tilde{u}_{n}^{2}(x)}\tilde{v}_{n}^{2}(x) \rightarrow \infty,\quad n\rightarrow \infty.$$
  Since $V(x)$ and $f(x,u)$ are both T-periodic in $x$, 
 $$
  \Phi(\tilde{u}_{n})=\Phi(u_{n}) \leq d
\text{ and } 
 \|\tilde{u}_{n}\| = \|u_{n}\| \rightarrow \infty \text{ as } n\rightarrow \infty.
 $$
 Note that
  \begin{equation} \label{**}
 	0<\dfrac{\Phi(\tilde{u}_{n})}{\|\tilde{u}_{n}\|^{2}} = \frac{1}{2}\|\tilde{v}_{n}^{+}\|^{2}-\frac{1}{2}\|\tilde{v}_{n}^{-}\|^{2}
 	-\int_{\mathbb{V}} \dfrac{F(x,\tilde{u}_{n})}{\tilde{u}_{n}^{2}(x)}\tilde{v}_{n}^{2}(x) d\mu.
 \end{equation}
This does not hold with $n$ big enough because $\tfrac{1}{2} \leq \|\tilde{v}_{n}^{+}\|^{2} \leq 1$.
\end{proof}
\begin{proposition}  \label{prop4}
	The mapping $\,\widehat{m} : E^{+}\backslash \{0\} \mapsto \mathcal{M}$ in Proposition \ref{prop2} is continuous.
\end{proposition}

\begin{proof}
	Let $u_{n} \rightarrow u$ in $E^{+} \backslash \{0\}$, we show that there exists a subsequence $\hat{m}(u_{n}) \rightarrow \hat{m}(u)$. We may assume that $\|u_{n}\| = \|u\| = 1$ and then, $\hat{m}(u_{n})=\|\hat{m}(u_{n})^{+}\|u_{n}+\hat{m}(u_{n})^{-}$. By Lemma \ref{lem5}, there is $R>0$ such that, for each $n \in \mathbb{N}$,
	\begin{equation}\label{key}
		\Phi(\hat{m}(u_{n})) = \sup_{\hat{E}(u_{n})} \Phi = \sup_{\hat{E}(u_{n})\cap B_{R} }\Phi \leq \sup_{u \in \hat{E}(u_{n})\cap B_{R} } \|u^{+}\|^{2} \leq R^{2}.
	\end{equation}
Since $\Phi$ is coercive on $\mathcal{M}$ by Proposition \ref{prop3}, $\{ \hat{m}(u_{n}) \}$ is bounded. Passing to a subsequence, we may assume that 
$$
t_{n}:=\| \hat{m}(u_{n})^{+} \| \rightarrow t \quad \text{ and  } \quad  \hat{m}(u_{n})^{-} \rightharpoonup u_{*}^{-} \text{ in }  E^{-},
$$
where
$0<\sqrt{2c} \leq t \leq \sup_{n}\|\hat{m}(u_{n})\|$.
Moreover, by Proposition \ref{prop2},
$$
\Phi(\hat{m}(u_{n})) \geq \Phi(t_{n}u_{n}+\hat{m}(u)^{-}) \rightarrow \Phi(tu+\hat{m}(u)^{-}) = \Phi(\hat{m}(u)).
$$
By Fatou's lemma and the weak lower semicontinuity of the norm, we have 
\begin{displaymath}
	\begin{aligned}
		&\Phi\big(\hat{m}(u)\big)
		 \leq \lim_{n \rightarrow \infty} \Phi \big(\hat{m}(u_{n})\big) \\
		 &= \lim_{n \rightarrow \infty}  \Big( \frac{1}{2}t_{n}^{2} - \frac{1}{2}\| \hat{m}(u_{n})^{-} \|^{2} - \int_{\mathbb{V}} F\big(x,t_{n}u_{n}+\hat{m}(u_{n})^{-}\big) d \mu \Big)\\
		 &\leq \frac{1}{2}t^{2} - \frac{1}{2}\| u_{*}^{-} \|^{2} - \int_{\mathbb{V}} F(x,tu+u_{*}^{-}) d \mu\\
		 &=\Phi\big(tu+u_{*}^{-} \big) \leq \Phi \big(\hat{m}(u)\big).
	\end{aligned}
\end{displaymath}
It follows that all the inequalities above must be equalities. Hence, $\| \hat{m}(u_{n})^{-}\| \rightarrow \| u_{*}^{-} \|$ and then, $\hat{m}(u_{n})^{-} \rightarrow u_{*}^{-}=\hat{m}(u)^{-},\quad  \hat{m}(u_{n}) \rightarrow \hat{m}(u)$  in $E$. 
\end{proof}

Define the functional $\hat{\Psi}:  E^{+} \setminus\{0\} \to \mathbb{R} $ as $\widehat{\Psi}(w) := \Phi(\widehat{m}(w))$
and $\Psi:=\widehat{\Psi } |_{S^{+}} $.  We know that, from Proposition \ref{prop4}, $\hat{\Psi}$ is continuous.

\begin{proposition}\label{prop 5}
	$\hat{\Psi} \in C^{1}\big( E^{+}\backslash\{0\},\mathbb{R} \big)$
	and 
	$$  \hat{\Psi } '(w)z= \frac{\| \hat{m}(w)^{+} \|}{\|w\|}\Phi '\big(\hat{m}(w) \big)z, 
	  \quad \forall w, z\in E^{+} , w \neq 0. $$
\end{proposition}

\begin{proof}
	For $ w, z\in E^{+} $, choose $\delta >0$ small enough such that, for any $ |t| < \delta $, $w_{t} := w+tz \in E^{+} \backslash \{0\} $. Set $u_{t} :=\hat{m}(w_{t}) \in \mathcal{M}$ and then, $u_{t}=s_{t}w_{t}+u_{t}^{-}$ where 
	$s_{t}=\frac{\|u_{t}^{+}\|}{\|w_{t}\|} >0.$ Moreover, Proposition \ref{prop4} implies that the mapping $s: (-\delta,\delta) \to \mathbb{R}, \, t\mapsto s_{t} $ is continuous.
	By Proposition \ref{prop2} and the mean value theorem, there exist $\tau_{t}, \eta_{t} \in (0,1)$ such that
	\begin{displaymath}
		\begin{aligned}
		    \hat{\Psi}(w_{t})-\hat{\Psi}(w) &= \Phi(u_{t})-\Phi(u)
		    =\Phi(s_{t}w_{t}+u_{t}^{-})-\Phi(s_{0}w+u^{-})\\
		    &\leq \Phi(s_{t}w_{t}+u_{t}^{-}) - \Phi(s_{t}w+u_{t}^{-}) \\
		    &=\Phi'\big(s_{t}[w+\tau_{t}(w_{t}-w)]+u_{t} \big) s_{t}(w_{t}-w)\\
		    &=s_{0}\Phi'(u)tz + o(t) \, \text{ as } t\to0
		\end{aligned}
	\end{displaymath}
and
		\begin{displaymath}
		\begin{aligned}
			\hat{\Psi}(w_{t})-\hat{\Psi}(w) 
			&\geq \Phi(s_{0}w_{t}+u^{-}) - \Phi(s_{0}w+u^{-}) \\
			&=\Phi'\big(s_{0}[w+\eta_{t}(w_{t}-w)]+u_{t} \big) s_{0}(w_{t}-w)\\
			&=s_{0}\Phi'(u)tz + o(t) \, \text{ as } t\to0.
		\end{aligned}
	\end{displaymath}
Hence, we have
$$
\partial_{z} \hat{\Psi }(w) = \lim_{t \to 0} \frac{\hat{\Psi }(w_{t})-\hat{\Psi }(w)}{t}=s_{0}\Phi'(u)z = \dfrac{\| \hat{m}(w)^{+} \|}{\| w \| } \Phi'(\hat{m}(w))z.
$$
Note that $ \partial_{z}\hat{\Psi }(w) $ is a continuous linear functional about $z$ and continuously depends on $w$. Therefore, the proposition holds.
\end{proof}
Consider$$ m:=\hat{m}|_{S^{+}} ,\quad \Psi = \Phi \circ m \equiv \widehat{\Psi} \big|_{S},$$ 
where $ S^{+}:=\{ u\in E^{+}: \|u\|=1 \}$. Then, $m:  S^{+} \to \mathcal{M}$ is homeomorphic with its inverse given by 
\begin{displaymath}
	\begin{aligned}
		\check{m}: \mathcal{M} &\to S^{+} \\
	                     	u&\mapsto \check{m}(u)=\frac{u^{+}}{\|u^{+}\|}.
	\end{aligned}
\end{displaymath}
Moreover, we have the following corollaries.

\begin{corollary}\label{cor6}
	(i) $\Psi \in C^{1} (S^{+} ) $ and for every $ w\in S^{+}$,
	$$
	\Psi'(w)z = \| \hat{m}(w)^{+} \| \Phi' \big( \hat{m}(w) \big)z, \quad
	\forall z\in T_{w}S^{+}:=\{ v\in E^{+}: \langle w,v \rangle=0 \}.
	$$
	(ii) $\{w_{n}\}_{n}$ is a Palais-Smale sequence for $\Psi$ if and only if $ \{ \hat{m}(w_{n})\}_{n} $ is a Palais-Smale sequence for $\Phi$ .\\
	(iii) We have $$
	\inf_{S^{+}} \Psi = \inf_{\mathcal{M}} \Phi =c,
	$$
	and $w\in S^{+}$ is a critical point of $\Psi$ if and only if  $\hat{m}(w) \in \mathcal{M}$ is a critical point of $\Phi$.
\end{corollary}

\begin{proof}
	(i) This is a direct corollary of Proposition \ref{prop 5}.\\
	(ii)Let $\{ w_{n} \} \subset S^{+}$ be a sequence such that $ C:= \sup_{n} \Psi(w_{n}) = \sup_{n} \Phi\big(u_{n}\big) < \infty$, where $ u_{n} := m(w_{n})\in \mathcal{M}$. For each $n$, we have an orthogonal splitting
	$$
	E=E^{-} \oplus \mathbb{R}(w_{n}) \oplus T_{w_{n}}S^{+} = E(w_{n}) \oplus T_{w_{n}}S^{+} = E(u_{n}) \oplus T_{w_{n}}S^{+} .
	$$
 $ u_{n} \in \mathcal{M} $ implies that $ \Phi'(u_{n})v=0$ for all $v\in E(u_{n})$ and then, $ \nabla \Phi (u_{n}) \in T_{w_{n}}S^{+}$. So,
$$
\| \Phi'(u_{n}) \| = \sup_{\substack{ z\in T_{w_{n}}S^{+} \\ \|z\|=1 }} \big| \Phi'(u_{n})z \big| .
$$
By (i), 
$$
\| \Psi'(w_{n}) \|
 = \sup_{\substack{ z\in T_{w_{n}}S^{+} \\ \|z\|=1  }} 
\big| \Psi'(w_{n})z \big|
=\sup_{\substack{ z\in T_{w_{n}}S^{+} \\ \|z\|=1 }}
 \| u_{n}^{+} \| \big| \Phi'(u_{n})z \big|.
$$
By Lemma \ref{lem4}(ii) and Proposition \ref{prop3}, $ 0<\sqrt{2c} \leq \|u_{n}^{+}\| \leq \|u_{n}\| < \infty $ for every $n\in\mathbb{N}$. Hence, $ \Psi'(w_{n}) \to 0 $ if and only if $ \Phi'(m(u_{n})) \to 0 $ as $ n\to \infty$.
(iii) One can prove by the definition of $\Psi$ and the same orthogonal spliting of $E$ in (ii).
\end{proof}

\section{The completed proof of Theorem 1.1 }
\begin{proof}[Proof of Theorem \ref*{thm1}]
	If $u_{0}\in \mathcal{M}$ satisfies $\Phi(u_{0})=c$, where $c:=\inf_{\mathcal{M}}	\Phi = \inf_{S^{+}} \Psi$, then $\check{m}(u_{0})\in S^{+}$ is a minimizer of $\Psi$ . Therefore, $\check{m}(u_{0})$ is a critical point of $\Phi$ and by Corollary \ref{cor6}(iii), $u_{0}$ is a critical point of $\Phi$.  Hence, it remains to show the existence of such $u_{0}$. 
	Let $\{w_{n}\} \subset S^{+}$ and $\Psi(w_{n}) \to c$, by Ekeland's Variational Principle, we may assume that $\Psi'(w_{n}) \to 0$ as $n\to \infty$. Put $u_{n}=m(w_{n})$, then $\Phi(u_{n})\to c$ and $\Phi'(u_{n}) \to 0$ as $n \to \infty$. By Proposition \ref*{prop3}, $\{u_{n}\}$ is bounded. We may assume that, after passing to a subsequence,
	$$u_{n} \rightharpoonup u$$
	and
	$$u_{n}(x)\to u(x), \quad \forall x \in \mathbb{V},$$
	as $n\to \infty$ by choosing a test function $\phi(y) = \delta_{x}(y)$.\\
	Suppose $u_{n}^{+} \to 0$ in $\ell^{p}(\mathbb{V}), \, p>2$. 
	Since $\|u\| \simeq \|u\|_{2}$ on $E$, by Lemma \ref{lem2},  we may choose $\epsilon >0 $ small enough such that 
	$$ \Phi'(u_{n})u_{n}^{+} 
	=\|u_{n}^{+}\|^{2} -
	 \int_{\mathbb{V}} f(x,u_{n})u_{n}^{+} d\mu 
    \geq  \frac{1}{2}\|u_{n}^{+}\|^{2} - 
             C_{\epsilon} \|u_{n}^{+}\|_{p}^{p}.
	$$
	Then, $\|u_{n}^{+}\| \to 0$ as $n \to \infty$, which contradicts
    Lemma \ref{lem4}(ii).
	Therefore, 
	$$\varlimsup_{n} \|u_{n}^{+}\|_{p} = \alpha,$$
	for some constant $\alpha >0 $. 
	By interpolation inequality, 
	$$
	\|u_{n}^{+}\|_{p} \leq \|u_{n}^{+}\|_{2}^{\frac{2}{p}} \|u_{n}^{+}\|_{\infty} ^{\frac{p-2}{p}}.
	$$
	Taking the upper limit on both sides, 
	$$
	\alpha=\varlimsup_{n} \|u_{n}^{+}\|_{p} \leq \varlimsup_{n} \|u_{n}^{+}\|_{2}^{\frac{2}{p}} \|u_{n}^{+}\|_{\infty}^{\frac{p-2}{p}}.
	$$
	Since $\{u_{n}\}$ is bounded, it follows that
	$$
	\varlimsup_{n} \|u_{n}^{+}\|_{\infty} \geq K
	$$
	for some constant $K>0$.
	That's to say, there exists a subsequence $\{u_{n}^{+}\}$ and a sequence $\{y_{n}\} \subset \mathbb{V}$ such that $|u_{n}^{+}(y_{n})| \geq K$ for each $n$.
 	For every $y_{n}$, let $k_{n}=(k_{n}^{1}, \cdots, k_{n}^{N}) \in  \mathbb{Z}^{N}$ be a vector such that $\{ y_{n}-k_{n}T\} \subset \Omega$, where $\Omega=[0,T)^{N} \cap \mathbb{V}$ is a finite subset.
 	By translations, we define $\tilde{u}_{n}(y) := u_{n}(y+k_{n}T)$ and then, for each $\tilde{u}_{n}$,
 	$$
 	\| \tilde{u}^{+}_{n} \|_{\ell^{\infty}(\Omega)} \geq |\tilde{u}_{n}^{+}(y_{n}-k_{n}T)| =
 	|u_{n}^{+}(y_{n})| \geq K.
 	$$
	Since $V$ and $f(x,u)$ are both T-periodic in $x$, one can check that $\tilde{u}_{n} \in \mathcal{M}$ and 
	$$
	\Phi(\tilde{u}_{n}) \to c, \quad  \Phi'(\tilde{u}_{n}) \to 0,\quad n \to \infty.
	$$
	Up to a subsequence, we may assume that 
	$$
	\tilde{u}_{n} \rightharpoonup \tilde{u} \text{ and  } 	\tilde{u}_{n}(x) \to \tilde{u}(x), \quad \forall x\in \mathbb{V}.
	$$
	Since $\Omega$ is finite, there exists $x_{0}\in \Omega$ such that $\tilde{u}_{n}^{+}(x_{0}) \to \tilde{u}^{+}(x_{0}) \neq 0$. \\
	Hence, $\tilde{u} \neq 0$.
 Moreover, for any $\varphi \in C_{c}(\mathbb{V})$, 
 \begin{displaymath}
 	\begin{aligned}
 		\langle \nabla \Phi(\tilde{u}_n), \varphi \rangle &=
 		\langle \tilde{u}_{n}, \varphi^{+} \rangle
 		-\langle \tilde{u}_{n}, \varphi^{-} \rangle
 		-\int_{\mathbb{V}} f(x,\tilde{u}_{n}) \varphi d\mu \\
 		& \to 	\langle \tilde{u}, \varphi^{+} \rangle
 		-\langle \tilde{u}, \varphi^{-} \rangle
 		-\int_{\mathbb{V}} f(x,\tilde{u}) \varphi d\mu \\
 		& = \langle \nabla \Phi(\tilde{u}),\varphi \rangle,
 	\end{aligned}
 \end{displaymath}
which implies $\Phi'(\tilde{u})=0$ and $\tilde{u} \in \mathcal{M}$.\\
It remains to show that $\Phi(\tilde{u}) = c$. On one hand, by Fatou's lemma,
\begin{displaymath}
	\begin{aligned}
		c+o(1) 
		&= \Phi(\tilde{u}_{n}) - \frac{1}{2} \Phi'(\tilde{u}_{n})\tilde{u}_{n}
		= \int_{\mathbb{V}} [\frac{1}{2}f(x,\tilde{u}_{n})\tilde{u}_{n} - F(x,\tilde{u}_{n})] d\mu \\
		& \geq \int_{\mathbb{V}}[\frac{1}{2}f(x,\tilde{u})\tilde{u} - F(x,\tilde{u})] d\mu  +  o(1) 
	    = \Phi(\tilde{u}) - \dfrac{1}{2}\Phi'(\tilde{u})\tilde{u}+o(1)\\
	     & = \Phi(\tilde{u}) +o(1) 
	     \quad \text{ as } n \to \infty.
	\end{aligned}
\end{displaymath}
On the other,
 $\Phi(\tilde{u}) \geq c = \inf_{\mathcal{M}} \Phi$. 
\end{proof}

\section{The idea of the proof of Theorem \ref{thm2}}
 At the very beginning, we recall the definition and some important properties of Krasnoselskii genus, see A. Szulkin and T. Weth's work \cite{MR1411681} for example. 
 \begin{definition}
 	Let $E$ be a Banach space. For all closed and symmetric nonempty subsets $A\subset E$, i.e., $ A=-A=\overline{A} \neq\emptyset$, define the Krasnoselskii genus as 
 	\begin{displaymath}
 		\gamma(A)=
 		\left\{
 		        \begin{aligned}
 		        	&\inf\{m\in\mathbb{N}_{+}: \exists h \in C^{0}(A;\mathbb{R}^{m} \backslash \{0\}) , h(-u)=-h(u)\}, \\
 		        	&\infty, \quad  \quad\text{if  } \{\cdots\}=\emptyset, \text{in particular, if }\, 0 \in A,
 		        \end{aligned}
 	    \right.
 	\end{displaymath}
  and define $\gamma(\emptyset)=0$.
   \end{definition}
  \begin{proposition}\label{genus}
  		Let $A,A_{1},A_{2}$ be closed and symmetric subsets of Banach space $E$, then\\
  	(i) $\gamma(A) \geq 0$ and $\gamma(A)=0$ if and only if $A=\emptyset$,\\
  	(ii) $\gamma(A_{1} \cup A_{2}) \leq \gamma(A_{1})+\gamma(A_{2})$,\\
  	(iii) if $f: A \to f(A)$ is odd and continuous, then $\gamma(A) \leq \gamma(\overline{f(A)})$,\\
  	(iv) if $A$ is compact, $0\notin A$ and $A\neq \emptyset$, then there exist open set $U \supset A$ such that $\gamma(\overline{U}) = \gamma(A)$.
  \end{proposition}
 We recite some notations introduced in \cite{MR2557725}. For $d\geq e \geq c$, set 
 \begin{displaymath}
 	\begin{aligned}
 	&\Phi^{d}:=\{u\in \mathcal{M}:\Phi(u)\leq d\},&
 	&\Phi_{e}:=\{u\in \mathcal{M}:\Phi(u)\geq e\},&
 	&\Phi^{d}_{e}:=\Phi^{d} \cap \Phi_{e},&\\
 		&\Psi^{d}:=\{w\in S^{+}:\Psi(w)\leq d\},&
 		&\Psi_{e}:=\{w\in S^{+}:\Psi(w)\geq e\},&
 		&\Psi^{d}_{e}:=\Psi^{d} \cap \Psi_{e},&\\
 		   &K:=\{ w\in S^{+}:\Psi'(w)=0 \},&
 	   	   &K_{d}:=\{ w\in K:\Psi(w)=d \},&
 	   	   &\nu_{d}:=\sup_{u \in \Phi^{d}} \|u\|.&
 	\end{aligned}
 \end{displaymath}
\begin{proof}[Sketch of proof of Theorem \ref{thm2}]
	By the periodicity of $f$ and $V$, one can check $K$ is symmetric w.r.t. the origin, i.e., $w\in K$ impies $-w\in K$. Choose a subset $\mathcal{F} \subset K $ such that
$ \mathcal{F} = -\mathcal{F}$ and each orbit $\mathcal{O}(w) \subset \mathcal{F} $ has a unique representative in  $\mathcal{F}.$ 
By C orollary \ref{cor6}, the orbits $\mathcal{O}(u) \subset \mathcal{M}$ consisting of critical points of $\Phi$ are in 1-1 correspondence with the orbits $\mathcal{O}(w)\subset S^{+}$ 
which contain critical points of $\Psi$. Therefore, it suffices to show $\mathcal{F}$ is infinite. 
Suppose by contradiction that 
\begin{equation}\label{suppose}
	\mathcal{F} \text{ is a finite set. }
\end{equation}

The purpose is to prove, for infinitely many different $d$,  $K_{d} \neq \emptyset$, which is equivalent to show $\gamma(K_{d})\neq 0$. Specifically, they consider the Lusternik-Schnirelman values for $\Psi$ defined by
$$ c_{k}:= \inf\{d\in \mathbb{R}: \gamma(\Psi^{d}) \geq k \},\quad k\in \mathbb{N},
$$
and claim that
\begin{equation}\label{claim}
	K_{c_{k}} \neq \emptyset \, \text{ and }\, c_{k}<c_{k+1}, \quad \forall k \in \mathbb{N}.
\end{equation}
Let $k\in \mathbb{N}$ and denote $d=c_{k}$. 
\begin{lemma}\label{lem2.1}
	$ \kappa := \inf \{ \|v-w\|: v, w \in K, v \neq w  \} >0.$
\end{lemma}
That's to say, the critical points are discretely distributed. So,
\begin{equation}\label{y<1}
	\gamma(K_{d})\leq \gamma(K) \leq 1.
\end{equation}
It follows immediately from the special choice of $d=c_{k}$ that, $\forall \epsilon>0$,
 $$\gamma(\Psi^{d-\epsilon}) \leq k-1 \text{ and } 
 \gamma(\Psi^{d+\epsilon}) \leq k.$$
By Proposition \ref{genus}(iv), there exists $U \supset K_{d}$ with $\gamma(\overline{U})=\gamma(K_{d})$. Specially, set $U:=U_{\delta}(K_{d})$ to ensure that 
$\Psi^{d+\epsilon} \backslash U$ is a closed and symmetric subset. If  futher we have 
\begin{equation}\label{mapping}
	\gamma(\Psi^{d+\epsilon} \backslash U) \leq \gamma(\Psi^{d-\epsilon})\leq k-1,
\end{equation}
then 
\begin{equation} \label{budengshi}
	k \leq \gamma(\Psi^{d+\epsilon})
		\leq \gamma(\Psi^{d+\epsilon} \backslash U)             +\gamma(\overline{U}) 
		\leq k-1+\gamma(K_{d})
\end{equation}
is enough to obtain (\ref{claim}). By Proposition \ref{genus}(iii), the key point is to establish an odd continuous mapping $h: \Psi^{d+\epsilon}\backslash U \to \Psi^{d-\epsilon}$ to obtain (\ref{mapping}). 

 For this purpose, they utilize the pseudo-gradient vector field of $\Psi$. Note that $\Psi \in C^{1}(S^{+})$, there exists a Lipschitz continuous mapping 
 \begin{displaymath}
 	\begin{aligned}
 		H:S^{+}\backslash K &\to TS^{+}\\
 		w &\mapsto H(w)\in T_{w}S^{+}
 	\end{aligned}
 \end{displaymath}
 and for all $w\in S^{+}\backslash K$,
 \begin{displaymath}
 	\left\{
 	\begin{aligned}
 		&\|H(w)\| < 2 \| \nabla \Psi(w) \|,   \\
 		&\langle H(w), \nabla \Psi(w) \rangle > \frac{1}{2} \| \nabla \Psi(w)\|^{2}.
 	\end{aligned}
 	\right.
 \end{displaymath}
 The corresponding pseudo-gradient flow 
 $\eta:\mathcal{G} \to S^{+}\backslash K$
 is defined by
 \begin{displaymath}
 	\left\{
 	\begin{aligned}
 		&\frac{d}{dt} \eta(t,w)=-H\big(\eta(t,w)\big),   \\
 		&\eta(0,w)=w,
 	\end{aligned}
 	\right.
 \end{displaymath}
 where
 $\mathcal{G}:=
 \{ (t,w): w\in S^{+}\backslash K, T^{-}(w)<t<T^{+}(w) \}$
 and $T^{-}(w)<0$, $T^{+}(w) >0$ are the maximal existence times of the trajectories of $\eta$. It worth pointing out that $\Psi$ is strictly decreasing along the trajectories of $\eta$.
 \begin{lemma}\label{lem2.3}
 	Let $d\geq c$. For every $\delta>0$, there exists $\epsilon = \epsilon(\delta)>0$ such that 
 	$(i)$ $\Psi_{d-\epsilon}^{d+\epsilon} \cap K = K_{d}$,\\
 	$(ii)$ $\lim_{t \to T^{+}(w)} \Psi \big( \eta(t,w) \big) <d- \epsilon$ for every $w \in \Psi^{d+\epsilon} \backslash U_{\delta}(K_{d})$.
 \end{lemma}
Briefly speaking, $(i)$ tells us the critial values are discretely distributed and $(ii)$ says every point in $\Psi^{d+\epsilon} \backslash U_{\delta}(K_{d})$ falls below the level set  $\Psi^{d-\epsilon}$ along the pseudo-gradient flow.
Based on this,
A. Szulkin and T. Weth define the entrance time mapping
$e: \Psi^{d+\epsilon}\backslash U \to [0,\infty )$ by
\begin{equation}
		e(w):= \inf\{ t\in \big[0,T^{+}(w)\big):\Psi\big(\eta(t,w)\big) \leq d-\epsilon \},
\end{equation}
where $U=U_{\delta}(K_{d})$, $\delta < \tfrac{\kappa}{2}$, i.e., $U\cap K = K_{d}$, and $\epsilon=\epsilon(\delta)>0$ is small enough to guarantee Lemma \ref{lem2.3}(ii). For small $\epsilon$, $d-\epsilon$ is not a critical value and this implies $e$ is continuously.  $\Psi$ is even and $\nabla \Psi$ is odd since $f$ is odd. So, we can choose $H$ to be odd and then $\eta$ is odd and $e$ is even. Hence, the mapping
$$h:\Psi^{d+\epsilon}\backslash U \to \Psi^{d-\epsilon},\quad
h(w):=\eta\big( e(w),w \big)$$
is odd and continuous and (\ref{mapping}) is obtained. It is easy to see  $\gamma(K_{d}) =1$ from (\ref{budengshi}) and (\ref{y<1}). Suppose $c_{k} = c_{k+1}$ for some $k\in\mathbb{N}$, then 
$\gamma (\Psi^{d})= \gamma (\Psi^{c_{k}+1})\geq k+1$ and $\gamma(K_{d})\geq 2$ which contradicts (\ref{y<1}), their claim (\ref{claim}) follows.  

By (\ref{claim}), there is an infinite many pairs of geometrically distinct critical points $\pm w_{k}$ with $\Psi(w_{k})=c_{k}$, contrary to (\ref{suppose}).
\end{proof}
Now, we have a look at the lemmas needed in the proof above. Most of the proof in discrete situation is similar to that in continuous case. However, there are indeed some differences. In continuous case, the authors have to use Sobolev embeddings to rescale the norm in different spaces and obtain strong convegence in a bounded domain. So, they need to restrict $p<2^{*}$. We can remove this restriction because of two reasons; weak convergence naturally implies pointwise convergence and $W^{1,2}(\mathbb{V})$ is equivalent to $\ell^{2}(\mathbb{V})$.
One can prove Lemma \ref{lem2.1} directly with the assumption of $\mathcal{F}$ being finite. 
To prove Lemma \ref{lem2.3}, we need the following for preparations.

\begin{lemma}\label{lem2.4}
	Let $d\geq c$. If $\{v_{n}^{1}\}_{n}, \{ v_{n}^{2} \}_{n} \subset \Psi^{d}$ are two Palais-Smale sequences for $\Psi$, then either $\|v_{n}^{1} - v_{n}^{2}\| \to \infty$ as $n\to \infty$ or $\varlimsup_{n\to \infty} \|v_{n}^{1} - v_{n}^{2}\| \geq \rho(d)>0$, where $\rho(d)$ depends on $d$ but not on the particular choice of Palais-Smale sequences.
\end{lemma}

By this lemma, the distance between any two essentially different Palais-Smale sequences in $\Psi^{d}$ has a uniform lower positive bound.Their proof distinguish two cases; 
Case 1: 
$\|u_{n}^{1} - u_{n}^{2}\|_{p} \to 0 $ as $n \to \infty$, where $u=\hat{m}(w)$. By scaling and calculation, they show 
$\|(u_{n}^{1}-u_{n}^{2})^{+}\| \to 0,
\|(u_{n}^{1}-u_{n}^{2})^{-}\| \to 0$ 
and hence $\|u_{n}^{1}-u_{n}^{2}\|\to 0$ as $n\to \infty$ .
Case 2: $\|u_{n}^{1} - u_{n}^{2}\|_{p} \nrightarrow 0 $ as $n \to \infty$. To show that $u_{n}^{1} \rightharpoonup u^{1},  u_{n}^{2}\rightharpoonup u_{2}$ with $u^{1} \neq u^{2}$, they imply the P.L.Lion's Lemma, see \cite{MR1400007} for example, and the fact that $\hat{m},\check{m},\nabla \Phi, \nabla \Psi$ are all equivariant under translations. Further more, they show 
$\varlimsup_{n\to \infty} \|v_{n}^{1} - v_{n}^{2}\| \geq \rho(d)>0$ by some geometric argument where $\rho(d)$ depends on $d$.

\begin{lemma}\label{lem2.5}
	For every $w \in S^{+} $, the limit $\lim_{t \to T^{+}(w)} \eta(t,w)$ exists and is a critical point of  $\Psi$.
\end{lemma}
This lemma is surprising and its proof is delicate. For $w\in S^{+}$, write $d:=\Psi(w)$. Their proof distinguish two cases as well; Case 1: $T^{+}(w) < \infty$. The authors find a contradiction with the continuation theorem of solutions. Case 2: $T^{+}(w) = \infty$. It suffices to show that
\begin{equation}\label{T}
	\forall \epsilon >0,\, \exists  t_{\epsilon}>0, \quad s.t. \quad
	 \| \eta(t,w)-\eta(t_{\epsilon},w)\| < \epsilon, \quad \forall t\geq t_{\epsilon}.
\end{equation}
Suppose by contradiction, $(\ref{T})$ is not true, they choose four special time nodes on the flow and find two Palais-Smale sequence $\{w_{n}^{1}\},\{w_{n}^{2}\} \subset \Psi^{d}$ with 
$0<\epsilon \leq \|w_{n}^{1}-w_{n}^{2}\| \leq \rho(d)$, which is contradicts Lemma \ref{lem2.4}.
\\ \hspace*{\fill}

Finally, we talk about the proof of Lemma \ref{lem2.3}. (i) is easy to see provided $\mathcal{F}$ is finite. By Lemma \ref{lem2.5}, if (ii) fails, then
\begin{equation}\label{9}
	\lim_{t\to T^{+}(w)} \Psi(\eta(t,w)) \in [d-\epsilon,d+\epsilon]
\end{equation}
for some $w\in \Psi^{d+\epsilon}\backslash U_{\delta}(K_{d})$. Combined with (i), we have 
\begin{equation}\label{10}
	\lim_{t\to T^{+}(w)} \Psi(\eta(t,w)) = \tilde{w} \in K_{d}.
\end{equation}
The purpose is to prove that before the time node $T^{+}(w)$, the flow $\eta(t,w)$ has fallen below the level set $d$. Then 
$$\lim_{t \to T^{+}(w)} \Psi(\eta(t,w))<d$$
and
$$
\eta(t,w) \nrightarrow \tilde{w} \text{ as } t \to T^{+}(w),
$$
contradicts to (\ref{10}).
 
\section{Appendix}
In this section, we give the proof of Proposition \ref{prop1}.
\begin{proof}[Proof of Proposition \ref{prop1}]
	For each given $ u \in \mathcal{M}$, set $\mathcal{W} := \{ su+v : s \geq -1, v \in E^{-} \}$, then each element in $\hat{E}(u)$ has the form $u+w$ with $w \in \mathcal{W}$. So, we only need to show $\Phi(u) > \Phi(u+w)$ for every $w \in \mathcal{W}, w\neq0$.\\
	Set
	$ B(v_{1},v_{2} ):=
	\int_{\mathbb{V}} \big( \Gamma(v_{1},v_{2})  +V(x)v_{1}v_{2}  \big) d\mu.$
	Calculate that
	\begin{displaymath}
		\begin{aligned}
			&\Phi(u+w)-\Phi(u)\\
			=&\tfrac{1}{2} \Big( B(u+w,u+w) -B(u,u) \Big)
			-\int_{\mathbb{V}}\Big(F(x,u+w)-F(x,u)\Big)d\mu \\
			=&\tfrac{1}{2} \Big( B\big((1+s)u+v,(1+s)u+v\big) -B(u,u) \Big)
			-\int_{\mathbb{V}}(F(x,u+w)-F(x,u))d\mu  \\
			=&\tfrac{1}{2}(s^{2}+2s)B(u,u)+(1+s)B(u,v)+\tfrac{1}{2}B(v,v) -\int_{\mathbb{V}}(F(x,u+w)-F(x,u))d\mu  \\
			=&-\tfrac{1}{2}\|v\|^{2} + B\big( u,\tfrac{1}{2} (s^{2}+2s)u+(1+s)v\big)-\int_{\mathbb{V}}(F(x,u+w)-F(x,u))d\mu \\
			=&-\tfrac{1}{2}\|v\|^{2} + \int_{\mathbb{V}} f(x,u)\big(\tfrac{1}{2}(s^{2}+2s)u+(1+s)v\big)+F(x,u)-F(x,u+w) d\mu
		\end{aligned}
	\end{displaymath}
	where the last equation is based on the fact that, since $u\in \mathcal{M}$,
	$$
	0=\Phi'(u)z=B(u,z)-\int_{\mathbb{V}} f(x,u)z \, d\mu, \quad \forall z \in E(u).
	$$
	Set $z(s):=u+w(s)=(1+s)u+v$ and
	$$g(s):= f(x,u)\big(\tfrac{1}{2}(s^{2}+2s)u+(1+s)v\big)+F(x,u)-F(x,z(s)).$$ 
	To prove $\Phi(u)>\Phi(u+w)$, it suffices to show $g(s)<0$ for any $s\geq -1$ such that $z(s)\neq-1$ .
	If $u=0$, then $g(s)=-F(x,z(s))<0$ for $z(s) \neq 0$. Otherwise, $u\neq 0$, we calculate that 
	\begin{displaymath}
		\begin{aligned}
			g(s)&=f(x,u)\big[\frac{1}{2}(s^{2}+2s)u+(s+1)v\big]+F(x,u)-F(x,z(s))\\&<f(x,u)\big[\frac{1}{2}(s^{2}+2s)u+(s+1)v\big]+\frac{1}{2}f(x,u)u-F(x,z(s))\\
	     &=\frac{1}{2}(s+1)^{2}f(x,u)u+(s+1)f(x,u)v-F(x,z(s))\\
	     &=\frac{1}{2}(s+1)^{2}f(x,u)u+(s+1)f(x,u)z-(s+1)^{2}f(x,u)u-F(x,z(s))\\
	     &=-\frac{1}{2}(s+1)^{2}f(x,u)u+(s+1)f(x,u)z-F(x,z(s)).
		\end{aligned}
	\end{displaymath}
By( \ref{uf(x,u)}), we have $g(s)<0$ whenever $uz \leq 0$. Note that
\begin{displaymath}
	\begin{aligned}
		g'(s)&=(s+1)f(x,u)u+f(x,u)v-f(x,z(s))u\\
		&=(s+1)f(x,u)u+f(x,u)[z-(s+1)u]-f(x,z)u\\
		&=uz\big( \dfrac{f(x,u)}{u}-\dfrac{f(x,z)}{z}\big).
	\end{aligned}
\end{displaymath}
So, if $uz>0$, $g'(s)$ is always positive (or negative)  and $g(s)$ is strictly increasing (or decreasing) on $[-1,+\infty)$.
Moreover, we have by $S_{5}$
$$g(-1)=-\dfrac{1}{2}f(x,u)u+F(x,u)-F(x,v)<-F(x,v)\leq 0	\text{ and }
\lim\limits_{s\rightarrow \infty}g(s) = -\infty.
$$ Thus, $g(s)<0$ for all $s\geq -1$.
\end{proof}

\textbf{Acknowledgements.}
 I would like to give my sincere thanks to my academic supervisor Prof. B.Hua for his invaluable instruction. Without his long-standing guidance and inspiration, this work could not have been completed. I would also like to  express my heartfelt thanks to J.Cheng and J.Wang for helpful discussions and suggestions. W.Xu is supported by Shanghai Science and Technology Program [Project No. 22JC1400100].
\bibliography{ref1}

\newcommand{\etalchar}[1]{$^{#1}$}
\begin{thebibliography}{XWYS85}

\bibitem[AP19]{MR3921107}
Setenay Akduman and Alexander Pankov.
\newblock Nonlinear {S}chr\"{o}dinger equation with growing potential on
  infinite metric graphs.
\newblock {\em Nonlinear Anal.}, 184:258--272, 2019.

\bibitem[Bar17]{MR3616731}
Martin~T. Barlow.
\newblock {\em Random walks and heat kernels on graphs}, volume 438 of {\em
  London Mathematical Society Lecture Note Series}.
\newblock Cambridge University Press, Cambridge, 2017.

\bibitem[BHL{\etalchar{+}}15]{MR3316971}
Frank Bauer, Paul Horn, Yong Lin, Gabor Lippner, Dan Mangoubi, and Shing-Tung
  Yau.
\newblock Li-{Y}au inequality on graphs.
\newblock {\em J. Differential Geom.}, 99(3):359--405, 2015.

\bibitem[BN83]{MR709644}
Ha\"{\i}m Br\'{e}zis and Louis Nirenberg.
\newblock Positive solutions of nonlinear elliptic equations involving critical
  {S}obolev exponents.
\newblock {\em Comm. Pure Appl. Math.}, 36(4):437--477, 1983.

\bibitem[Cao92]{MR1163431}
D.~M. Cao.
\newblock Nontrivial solution of semilinear elliptic equation with critical
  exponent in {${\bf R}^2$}.
\newblock {\em Comm. Partial Differential Equations}, 17(3-4):407--435, 1992.

\bibitem[CZR92]{MR1181725}
Vittorio Coti~Zelati and Paul~H. Rabinowitz.
\newblock Homoclinic type solutions for a semilinear elliptic {PDE} on {${\bf
  R}^n$}.
\newblock {\em Comm. Pure Appl. Math.}, 45(10):1217--1269, 1992.

\bibitem[GLY16a]{MR3523107}
Alexander Grigor'yan, Yong Lin, and Yunyan Yang.
\newblock Kazdan-{W}arner equation on graph.
\newblock {\em Calc. Var. Partial Differential Equations}, 55(4):Art. 92, 13,
  2016.

\bibitem[GLY16b]{MR3542963}
Alexander Grigor'yan, Yong Lin, and Yunyan Yang.
\newblock Yamabe type equations on graphs.
\newblock {\em J. Differential Equations}, 261(9):4924--4943, 2016.

\bibitem[GLY17]{MR3665801}
Alexander Grigor'yan, Yong Lin, and YunYan Yang.
\newblock Existence of positive solutions to some nonlinear equations on
  locally finite graphs.
\newblock {\em Sci. China Math.}, 60(7):1311--1324, 2017.

\bibitem[Gy18]{MR3822363}
Alexander Grigor'yan.
\newblock {\em Introduction to analysis on graphs}, volume~71 of {\em
  University Lecture Series}.
\newblock American Mathematical Society, Providence, RI, 2018.

\bibitem[HJ14]{MR3265330}
Bobo Hua and J\"{u}rgen Jost.
\newblock {$L^q$} harmonic functions on graphs.
\newblock {\em Israel J. Math.}, 202(1):475--490, 2014.

\bibitem[HL17]{MR3581303}
Bobo Hua and Yong Lin.
\newblock Stochastic completeness for graphs with curvature dimension
  conditions.
\newblock {\em Adv. Math.}, 306:279--302, 2017.

\bibitem[HLLY19]{MR4036571}
Paul Horn, Yong Lin, Shuang Liu, and Shing-Tung Yau.
\newblock Volume doubling, {P}oincar\'{e} inequality and {G}aussian heat kernel
  estimate for non-negatively curved graphs.
\newblock {\em J. Reine Angew. Math.}, 757:89--130, 2019.

\bibitem[HSZ20]{MR4053610}
Xiaoli Han, Mengqiu Shao, and Liang Zhao.
\newblock Existence and convergence of solutions for nonlinear biharmonic
  equations on graphs.
\newblock {\em J. Differential Equations}, 268(7):3936--3961, 2020.

\bibitem[KLW21]{MR4383783}
Matthias Keller, Daniel Lenz, and Rados\l aw~K. Wojciechowski.
\newblock {\em Graphs and discrete {D}irichlet spaces}, volume 358 of {\em
  Grundlehren der mathematischen Wissenschaften [Fundamental Principles of
  Mathematical Sciences]}.
\newblock Springer, Cham, [2021] \copyright 2021.

\bibitem[Lio85]{MR850686}
P.-L. Lions.
\newblock The concentration-compactness principle in the calculus of
  variations. {T}he limit case. {II}.
\newblock {\em Rev. Mat. Iberoamericana}, 1(2):45--121, 1985.

\bibitem[LMP18]{MR3776357}
Shiping Liu, Florentin M\"{u}nch, and Norbert Peyerimhoff.
\newblock Bakry-\'{E}mery curvature and diameter bounds on graphs.
\newblock {\em Calc. Var. Partial Differential Equations}, 57(2):Paper No. 67,
  9, 2018.

\bibitem[LWZ06]{MR2271695}
Yongqing Li, Zhi-Qiang Wang, and Jing Zeng.
\newblock Ground states of nonlinear {S}chr\"{o}dinger equations with
  potentials.
\newblock {\em Ann. Inst. H. Poincar\'{e} C Anal. Non Lin\'{e}aire},
  23(6):829--837, 2006.

\bibitem[Man20]{MR4092834}
Shoudong Man.
\newblock On a class of nonlinear {S}chr\"{o}dinger equations on finite graphs.
\newblock {\em Bull. Aust. Math. Soc.}, 101(3):477--487, 2020.

\bibitem[Pan89]{MR1026427}
A.~A. Pankov.
\newblock Semilinear elliptic equations in {${\bf R}^n$} with nonstabilizing
  coefficients.
\newblock {\em Ukrain. Mat. Zh.}, 41(9):1247--1251, 1295, 1989.

\bibitem[Pan05]{MR2175045}
A.~Pankov.
\newblock Periodic nonlinear {S}chr\"{o}dinger equation with application to
  photonic crystals.
\newblock {\em Milan J. Math.}, 73:259--287, 2005.

\bibitem[Rab91]{MR1205391}
Paul~H. Rabinowitz.
\newblock A note on a semilinear elliptic equation on {${\bf R}^n$}.
\newblock In {\em Nonlinear analysis}, Sc. Norm. Super. di Pisa Quaderni, pages
  307--317. Scuola Norm. Sup., Pisa, 1991.

\bibitem[Rab92]{MR1162728}
Paul~H. Rabinowitz.
\newblock On a class of nonlinear {S}chr\"{o}dinger equations.
\newblock {\em Z. Angew. Math. Phys.}, 43(2):270--291, 1992.

\bibitem[Str96]{MR1411681}
Michael Struwe.
\newblock {\em Variational methods}, volume~34 of {\em Ergebnisse der
  Mathematik und ihrer Grenzgebiete (3) [Results in Mathematics and Related
  Areas (3)]}.
\newblock Springer-Verlag, Berlin, second edition, 1996.
\newblock Applications to nonlinear partial differential equations and
  Hamiltonian systems.

\bibitem[SW09]{MR2557725}
Andrzej Szulkin and Tobias Weth.
\newblock Ground state solutions for some indefinite variational problems.
\newblock {\em J. Funct. Anal.}, 257(12):3802--3822, 2009.

\bibitem[SW10]{MR2768820}
Andrzej Szulkin and Tobias Weth.
\newblock The method of {N}ehari manifold.
\newblock In {\em Handbook of nonconvex analysis and applications}, pages
  597--632. Int. Press, Somerville, MA, 2010.

\bibitem[Ura00]{MR1782338}
Hajime Urakawa.
\newblock The spectrum of an infinite graph.
\newblock {\em Canad. J. Math.}, 52(5):1057--1084, 2000.

\bibitem[Wil96]{MR1400007}
Michel Willem.
\newblock {\em Minimax theorems}, volume~24 of {\em Progress in Nonlinear
  Differential Equations and their Applications}.
\newblock Birkh\"{a}user Boston, Inc., Boston, MA, 1996.

\bibitem[Woe00]{MR1743100}
Wolfgang Woess.
\newblock {\em Random walks on infinite graphs and groups}, volume 138 of {\em
  Cambridge Tracts in Mathematics}.
\newblock Cambridge University Press, Cambridge, 2000.

\bibitem[XWYS85]{MR840880}
Daoxing Xia, Zhuo~Ren Wu, Shao~Zong Yan, and Wu~Chang Shu.
\newblock {\em Shibian hanshu lun yu fanhan fenxi. {V}ol. {II}}.
\newblock Gaodeng Xuexiao Jiaocai. [Educational Materials for Advanced
  Schools]. Renmin Jiaoyu Chubanshe, Beijing, second edition, 1985.

\bibitem[Yan12]{MR2873855}
Yunyan Yang.
\newblock Existence of positive solutions to quasi-linear elliptic equations
  with exponential growth in the whole {E}uclidean space.
\newblock {\em J. Funct. Anal.}, 262(4):1679--1704, 2012.

\bibitem[ZZ18]{MR3833747}
Ning Zhang and Liang Zhao.
\newblock Convergence of ground state solutions for nonlinear {S}chr\"{o}dinger
  equations on graphs.
\newblock {\em Sci. China Math.}, 61(8):1481--1494, 2018.

\end{thebibliography}
\bibliographystyle{alpha}

\end{document}